\documentclass[draft, reqno]{amsart}
\usepackage[all]{xy}                        %

\CompileMatrices                            

\UseTips                                    

\input xypic
\usepackage[bookmarks=true]{hyperref}       

\usepackage{amssymb,latexsym,amsmath,amscd}
\usepackage{xspace}
\usepackage{enumerate}
\usepackage{graphicx}
\usepackage{vmargin}
\usepackage{todonotes}
\usepackage{xcolor}

\reversemarginpar

\theoremstyle{plain}
\newtheorem{theorem}{Theorem}[section]
\newtheorem*{theorem*}{Theorem}
\newtheorem{proposition}[theorem]{Proposition}
\newtheorem{corollary}[theorem]{Corollary}
\newtheorem{lemma}[theorem]{Lemma}

\newtheorem{claim}{Claim}

\theoremstyle{definition}
\newtheorem{definition}[theorem]{Definition}

\newtheorem{notation}[theorem]{Notation}

\newtheorem{remark}[theorem]{Remark}

\newcommand{\enm}[1]{\ensuremath{#1}}          %

\newcommand{\cal}[1]{\mathcal{#1}}

\newcommand{\NN}{\enm{\mathbb{N}}}

\newcommand{\PP}{\enm{\mathbb{P}}}

\newcommand{\Ii}{\enm{\cal{I}}}

\newcommand{\Oo}{\enm{\cal{O}}}

\newcommand{\Ss}{\enm{\cal{S}}}

\renewcommand{\phi}{\varphi}
\renewcommand{\theta}{\vartheta}
\renewcommand{\epsilon}{\varepsilon}

\begin{document}

\title[Third secant variety of Segre-Veronese]{On the ranks of the third secant variety of Segre-Veronese embeddings}

\author[E. Ballico]{Edoardo Ballico}
\address[Edoardo Ballico]{Dipartimento di Matematica,  Univ. Trento, Italy}
\email{edoardo.ballico@unitn.it }

\author[A. Bernardi]{Alessandra Bernardi}
\address[Alessandra Bernardi]{Dipartimento di Matematica, Univ.  Trento,  Italy}
\email{alessandra.bernardi@unitn.it}

\maketitle

\begin{abstract} We give an upper bound for the rank of the border rank 3 partially symmetric tensors. In the special case of border rank 3 tensors $T\in V_1\otimes \cdots \otimes V_k$  (Segre case) we can show that all ranks among 3 and $k-1$ arise and if $\dim V_i \geq 3$ for all $i$'s, then also all the ranks between $k$ and $2k-1$ arise.
\end{abstract}

\section*{Introduction}

In this paper we deal with the problem of finding a bound for the minimum integer $r(T)$ needed to write a given tensor $T$ as a linear combination of $r(T)$  decomposable tensors. Such a minimum number  is now known under the name of \emph{rank of $T$}. In order to be as general as possible we will consider the tensor $T$ to be partially symmetric, i.e.
\begin{equation}\label{PST}T\in S^{d_1} V_1 \otimes \cdots  \otimes S^{d_k} V_k \end{equation}
where the $d_i$'s are positive integers  and  $V_i$'s are finte dimensional vector spaces defined over an algebraically closed field $K$. The decomposition that will give us the rank of such a tensor will be of the following type:
\begin{equation}\label{rankT}T=\sum_{i=1}^{r(T)} \lambda_i v_{1,i}^{\otimes d_1} \otimes \cdots \otimes v_{k,i}^{\otimes d_k}\end{equation}
where  $\lambda_i \in K$  and $v_{j,i}\in V_j$, $i=1, \ldots, r(T)$ and  $j=1, \ldots , k$.

Another very interesting and useful notion of ``~rank~'' is the minimum $r(T)$ such that a tensor $T$ can be written as a limit of a sequence of rank $r(T)$ tensors. This last integer is called the {\emph{border rank of $T$}} (Definition \ref{DefBorderRank}) and clearly it can be strictly smaller than the rank of $T$ (Remark \ref{br<r}). It has become a common technique to fix a class of tensors of given border rank  and then study all the possible ranks arising in that family (cf. \cite{bgi,bb2,bl2,lt,bdhm}). The rank of tensors of border rank 2 is well known (cf. \cite{bgi} for symmetric tensors, \cite{bb1} for tensors without any symmetry, \cite{bb5} for  partially symmetric tensors). The first not completely classified case is the one of border rank 3 tensors. In \cite[Theorem 37]{bgi} the rank of any symmetric order $d$ tensor of border rank 3 has been computed and it is shown that  the maximum rank reached is $2d-1$. In the present paper, Theorem \ref{i1}, we prove that the rank of partially symmetric tensors $T$ as in (\ref{PST}) of border rank 3 can be at most $$r(T)\leq-1 +\sum _{i=1}^{k} 2d_i.$$
In \cite[Theorem 1.8]{bl2} J. Buczy\'nski and J.M. Landsberg described the cases in which the inequality in Theorem \ref{i1} is an equality: when $k=3$ and $d_1=d_2=d_3=1$  they show that there is an element  
of rank 5.
All ranks for border rank 3 partially symmetric tensors 
 are described in \cite{bl1} when $k=3$, $d_1=d_2=d_3=1$ and $n_i=1$ for at least one integer $i$. Therefore our Theorem \ref{i1} is the natural extension of the two extreme cases (tensors without any symmetry  where $d_i=1$ for all $i=1, \ldots , k$ and totally symmetric case where $k=1$). 

In the  special case of tensors without any symmetry, i.e.
$T\in V_1\otimes \cdots \otimes V_k$, we will be able to show, in Theorem \ref{dd1},  that all ranks among 3 and $k-1$ arise and if $\dim V_i \geq 3$ for all $i$'s then also all the ranks between $k$ and $2k-1$ arise, therefore this result is sharp (cf. Remark \ref{uu6}).
In the proof of this theorem we  will describe the structure of our solutions: they are all obtained  from $(\PP^2)^k$ taking as a border scheme a degree
$3$ connected curvilinear scheme. The two critical cases are rank $k-1$ on $(\mathbb{P}^1)^k$ and rank $2k-1$ on $(\mathbb{P}^2)^k$ and the other cases can be deduced from one of these two.

In \cite{bb4} we defined the notion of curvilinear rank for symmetric tensors to be the minimum length of a curvilinear scheme whose span contains a given symmetric tensor. We can extend some of the ideas in \cite{bb4} and some of those used in our proof of Theorem \ref{i1} to the case of partially symmetric tensors
 and prove that, if a partially symmetric tensor is contained in the span of a special degree $c$ curvilinear scheme with $\alpha$ components, the rank of this tensor is bounded by $2\alpha +c\left(-1+\sum _{i=1}^{k} d_i\right)$ (cf. Theorem \ref{i2}).

\section{Notation, Definitions and Statements}
In this section we introduce the basic geometric tools that we will use all along the paper.
\begin{notation}\label{Setting}
We indicate with
 $$\nu_{n_1^1, \ldots , n_k^1} : \PP^{n_1}\times \cdots \times \PP ^{n_k}\to \PP^M, \hbox{ where } M=\left(\prod _{i=1}^{k}(n_i+1)\right)-1$$ the \emph{Segre embedding} of the multi-projective
space $\PP^{n_1}\times \cdots \times \PP ^{n_k}$, i.e. the embedding of $\PP^{n_1}\times \cdots \times \PP ^{n_k}$ by the complete linear system $|\Oo _{\PP^{n_1}\times \cdots \times \PP ^{n_k}}(1,\dots ,1)|$. 

\bigskip

For each $i\in \{1,\dots ,k\}$ 
let $$\pi _i: {\PP^{n_1}\times \cdots \times \PP ^{n_k}}\to \PP^{n_i}$$
denote the projection onto the $i$-th factor.

\bigskip 

Let  
$$\tau _i: {\PP^{n_1}\times \cdots \times \PP ^{n_k}}\to \PP^{n_1}\times \cdots\times \hat \PP^{n_{i}}\times\cdots \times \PP^{n_k} $$ 
denote the  projection onto all the factors different from $\PP^{n_i}$.  

\bigskip

Let $\epsilon _i\in \NN^k$ be the $k$-tuple of integers
$\epsilon_i=(0,\dots, 1, \dots ,0)$ with  $1$ only in the  $i$-th position.
We say that \emph{a curve $C\subset {\PP^{n_1}\times \cdots \times \PP ^{n_k}}$ has multi-degree} $(a_1,\dots ,a_k)$ if for all $i=1,\dots ,k$ the line bundle $\Oo _C(\epsilon _i)$ has degree $a_i$.

\bigskip 

We say that \emph{a morphism $h: \PP^1\to {\PP^{n_1}\times \cdots \times \PP ^{n_k}}$ has multi-degree} $(a_1,\dots ,a_k)$ if, for all $i=1, \ldots , k$: 
$$h^\ast \left(\Oo _{\PP^{n_1}\times \cdots \times \PP ^{n_k}}(\epsilon _i)\right) \cong \Oo _{\PP^1}(a_i).$$

\bigskip

Let $$\nu _{n_1^{d_1},\dots ,n_k^{d_k}}: {\PP^{n_1}\times \cdots \times \PP ^{n_k}} \to \PP^N, \hbox{ where } N=\left(\prod _{i=1}^{k} \binom{d_i+n_i}{n_i}\right)-1$$ denote the Segre-Veronese embedding of ${\PP^{n_1}\times \cdots \times \PP ^{n_k}}$ of multi-degree $(d_1,\dots ,d_k)$ and define 
$$X:= \nu _{n_1^{d_1},\dots ,n_k^{d_k}}({\PP^{n_1}\times \cdots \times \PP ^{n_k}})$$ to be the \emph{Segre-Veronese variety}.
\end{notation}

The name ``~Segre-Veronese~" is classically due to the fact that when the $d_i$'s are all equal to 1, then the variety $X$ is called ``~Segre variety~"; while when $k=1$ then $X$ is known to be a ``~Veronese variety~".

\begin{remark} An element of $X$ is the projective class of a decomposable partially symmetric tensor $T\in S^{d_1}V_1 \otimes \cdots \otimes S^{d_k}V_k$ where $\PP(V_i)=\mathbb{P}^{n_i}$. More precisely $p\in X$ if there exists $T\in S^{d_1}V_1 \otimes \cdots \otimes S^{d_k}V_k$ such that  $p=[T]=[v_{1,i}^{\otimes d_1} \otimes \cdots \otimes v_{k,i}^{\otimes d_k}]$ with $[v_{j,i}]\in \mathbb{P}^{n_i}$.
\end{remark} 

\begin{definition}\label{SecantTangent} The \emph{$s$-th secant variety} $\sigma_s(X)$  of  $X$ is the Zariski closure of the union of all $s$-secant $\mathbb{P}^{s-1}$  to $X$. The \emph{tangential variety} $\tau(X)$ is the  Zariski closure of the union of all tangent lines to $X$.
\end{definition}

Observe that
\begin{equation}\label{filtrationSec}
X=\sigma_1(X)\subset \tau(X) \subset \sigma_2(X) \subset \cdots \subset \sigma_{s-1}(X)\subset \sigma_s(X) \subset \cdots \subset \mathbb{P}^N.
\end{equation}

\begin{definition}\label{DefRank}The \emph{$X$-rank} $r_X(p)$ of an element $p\in \mathbb{P}^N$ is the minimum integer $s$ such that there exist a $\mathbb{P}^{s-1}\subset \mathbb{P}^N$ which is $s$-secant  to $X$ and containing $p$. 

We indicate with $\mathcal{S}(p)$ the set of sets of points of $\PP^{n_1}\times \cdots \times \PP ^{n_k}$ ``~evincing~" the $X$-rank of $p\in \PP^N$, i.e. 
{\small{$$\mathcal{S}(p):=\left\{\{x_1, \ldots , x_s\}\subset \PP^{n_1}\times \cdots \times \PP ^{n_k} \, | \, r_X(p)=s \hbox{ and } p\in\langle \nu_{n_1^{d_1}, \ldots , n_k^{d_k}}(x_1), \ldots , \nu_{n_1^{d_1}, \ldots , n_k^{d_k}}(x_s)\rangle \right\}.$$}}
\end{definition}
\begin{definition}\label{DefBorderRank}
The \emph{$X$-border rank} $br_X(p)$ of an element $p\in \mathbb{P}^N$ is the minimum integer $s$ such that $p\in \sigma_s(X)$.
\end{definition}

\begin{remark}\label{br<r} For any $p\in \mathbb{P}^N=\mathbb{P}(S^{d_1}V_1 \otimes \cdots \otimes S^{d_k}V_k)$ we obviously have that $br_X(p)\leq r_X(P)$. In fact  $p\in \mathbb{P}^N$ of rank $r$ is such that there exist a tensor $T\in S^{d_1}V_1 \otimes \cdots \otimes S^{d_k}V_k$ that can be minimally written as in (\ref{rankT}); while an element  $p\in \mathbb{P}^N$ has border rank $s$ if and only if there exist a sequence of rank $r$ tensors $T_i\in S^{d_1}V_1 \otimes \cdots \otimes S^{d_k}V_k$ such that  $\lim_{i\rightarrow \infty} T_i=T$ and $p=[T]$.

\end{remark} 
The first result that we prove in Section \ref{Sectioni1} is an upper bound for the rank of points in $\sigma _3(X)$.
\begin{theorem}\label{i1}
The rank of an element $p\in \sigma _3(X)$ is $r_X({p})\le -1 +\sum _{i=1}^{k} 2d_i$.
\end{theorem}

In the case $d_i=1$ for all $i=1, \ldots ,k$, i.e.  if $X$ is the Segre variety, we fill in all low ranks with points of $\sigma _3(X)\setminus \sigma _2(X)$. In Section \ref{SectionSegre} we prove the following result.

{
\begin{theorem}\label{dd1}
Assume $k\ge 3$ and let $X\subset \PP^M$ be the Segre variety of $k$ factors as in Notation \ref{Setting}. Let $\alpha$ be the cardinality of the set $\{i\in \{1,\dots ,k\}\mid n_i\ge 2$\}.  For each $x\in \{3,\dots ,\alpha +k-1\}$ there is $p\in \sigma _3(X)\setminus \sigma _2(X)$
with $r_X({p}) =x$.
\end{theorem}}

\begin{remark} If $\alpha =k$, i.e. if $n_i \ge 2$ for all $i$'s, Theorems \ref{i1}  and \ref{dd1} give the ranks of the points of $\sigma _3(X)\setminus \sigma _2(X)$. 		
 In Remark \ref{uu6} we discuss the reason why we do not know the rank of  a specific $p\in \sigma _3(X)\setminus \sigma _2(X)$.

Moreover, in the case of Segre varieties where factors have dimension $n_i \geq 2$, Theorem \ref{dd1} says that all ranks from $3$ to $2k- 1$ can be attained. Therefore the above result is  sharp.
\end{remark}

As remarked in the Introduction, we can extend some of the ideas of \cite{bb4} on the notion of curvilinear rank to some tools  used in our proof of Theorem \ref{i1} to the case of partially symmetric tensors.

\begin{definition} A scheme $Z\subset {\PP^{n_1}\times \cdots \times \PP ^{n_k}}$ is \emph{curvilinear} if it is a finite union of disjoint schemes of the form $\mathcal{O}_{C_i,P_{i}}/{{m}}_{p_i}^{e_i}$ for smooth points $p_i\in \PP^{n_1}\times \cdots \times \PP ^{n_k}$ on reduced curves $C_i\subset\PP^{n_1}\times \cdots \times \PP ^{n_k}$. Equivalently $Z$ is curvilinear if the tangent space at each of its connected component  supported at the $p_i$'s has Zariski dimension $\leq 1$. We define the \emph{curvilinear rank} $\mathrm{Cr}(p)$  of a point $p\in \mathbb{P}^N$ as:
$$\mathrm{Cr}(p):=\min\left\{\deg(Z)\; | \; \nu _{n_1^{d_1},\dots ,n_k^{d_k}}(Z)\subset X, \; Z \hbox{ curvilinear, } p\in \langle \nu _{n_1^{d_1},\dots ,n_k^{d_k}}(Z) \rangle\right\}.$$
\end{definition}
 In Section \ref{SectionCurvilinear} we prove the following result.

 \begin{theorem}\label{i2}
 If there exists a  degree $c$ curvilinear scheme $Z\subset {\PP^{n_1}\times \cdots \times \PP ^{n_k}}$ such that $p\in \langle \nu _{n_1^{d_1},\dots ,n_k^{d_k}}(Z)\rangle$ and $Z$ has $\alpha$ connected components,
 each of them mapped by $\nu _{n_1^{d_1},\dots ,n_k^{d_k}}$ into a linearly independent zero-dimensional sub-scheme of $\PP^N$, then $r_X({p}) \le 2\alpha +c\left(-1+\sum _{i=1}^{k} d_i\right)$.
 \end{theorem}

\section{Proof of Theorem \ref{i1}}\label{Sectioni1}

\begin{remark}\label{Remarka1}
Fix a degree $3$ connected curvilinear scheme $E\subset \PP^2$ not contained in a line and a point $u\in \PP^1$. The scheme $E$ is contained in a smooth conic. Hence there is
an embedding $f: \PP^1 \to \PP^2$ with $f(\PP^1) =C$ and $f(3u) =E$.
\end{remark}

\begin{remark}\label{Remarka2}
For any couple of points $u, o\in \PP^1$, there is an isomorphism $f: \PP^1\to \PP^1$ such that $f(u) =o$. For any such $f$ we have $f(3u) =3o$.
\end{remark}

\begin{remark}\label{Remarka3}
Fix two points $u, o\in \PP^1$. There is a morphism $f: \PP^1\to \PP^1$ such that $f(u) =o$, $f$ is ramified at $u$ and $\deg (f)=2$, i.e. $f^\ast (\Oo _{\PP^1}(1)) \cong \Oo _{\PP^1}(2)$.
Since $\deg (f) =2$, $f$ has only order $1$ ramification at $u$. Thus $f(3u) = 2o$ (as schemes).
\end{remark}

We recall the following lemma proved in \cite[Lemma 3.3]{bb5}.

\begin{lemma}[Autarky]\label{autarky} Let $p\in\langle X \rangle$ with $X$ being the Segre-Veronese variety of $\PP^{n_1}\times \cdots \times \PP^{n_k}$ embedded in multidegree $(d_1, \ldots , d_k)$. If there exist $\mathbb{P}^{m_i}$, $i=1, \ldots , k$, with $m_i\leq n_i$, such that $p\in \langle \nu_{m_1^{d_1}, \ldots , m_k^{d_k}}(\PP^{m_1}\times \cdots \times \PP^{m_k})\rangle $, then the $X$-rank of $p$  is the same as the $Y$-rank of $p$ where $Y$ is the Segre-Veronese $\nu_{m_1^{d_1}, \ldots ,m_k^{d_k}}(\PP^{m_1}\times \cdots \times \PP^{m_k})$.
\end{lemma}

\begin{corollary}\label{AutarkyCorollary} Let  $\Gamma\subset \PP^{n_1} \times \cdots \times \PP^{n_k}$ be a 0-dimensional scheme of minimal degree such that  $p\in \langle \nu_{n_1^{d_1},\dots ,n_k^{d_k}} (\Gamma)\rangle$, then the $X$-rank of $p$ is equal to its $Y$-rank where $Y$ is the Segre-Veronese embedding of $ \PP^{m_1} \times \cdots \times\PP^{m_k}$ where each $m_i=\dim \langle \pi_i(\Gamma) \rangle-1\leq\deg ( \pi_i(\Gamma))-1$ ($\pi_i$  as in Notation \ref{Setting}). If there exists an index $i$ such that $\deg ( \pi_i(\Gamma))=1$, then we can take $Y$ to be the Segre-Veronese embedding of  $\PP^{m_1} \times \cdots \times \hat\PP^{m_{i}}\times \cdots \times \PP^{m_k}$.
\end{corollary}

\begin{proof} Consider the projections $\pi _i: {\PP^{n_1}\times \cdots \times \PP ^{n_k}}$ onto the $i$-th factor $\PP^{n_i}$ as in Notation \ref{Setting}. It may happen that $\deg(\pi_i(\Gamma))$ can be any value from $1$ to  $\deg(\Gamma)$.

By the just recalled Autarky Lemma (cf. Lemma \ref{autarky}), we may assume that each $\pi _i(\Gamma)$ spans the whole $\PP^{n_i}$. Therefore if there is an index $i\in \{1, \ldots , k\}$ such that $\deg(\pi_i(\Gamma))=1$ we can take  $p\in \PP^{n_1} \times \cdots \times \hat\PP^{n_{i}}\times \cdots \times \PP^{n_k}$. 
Moreover the autarkic fact that we can assume  $\PP^{n_i}$ to be $\langle\pi_i(\Gamma)\rangle$ implies that we can replace each $\PP^{n_i}$ with $\PP^{\dim\langle\pi_i(\Gamma)\rangle-1}$ and clearly $\dim\langle\pi_i(\Gamma)\rangle\leq \deg (\pi_i(\Gamma))$.
\end{proof}

\begin{proof}[Proof of Theorem \ref{i1}:] Because of the filtration of secants varieties \eqref{filtrationSec}, for a given element $p\in \sigma_3(X)$, it may happen that either $p\in X$, or $p\in \sigma_2(X)\setminus X$ or $p\in \sigma_3(X)\setminus \sigma_2(X)$. We distinguish among these cases.

\begin{enumerate}

\item If $p\in X$, then $r_X({p})=1$. 

\item If $p\in \sigma _2(X)\setminus X$ then either $p$ lies on a honest bisecant line to $X$ (and in this case obviously $r_X({p}) =2$) or $p$ belongs to certain tangent line to $X$.
In this latter case,   the minimum number $h\leq k$ of factors containing such a tangent line is the minimum integer such that $p\in \langle \nu_{n_1^{d_1},\dots ,n_h^{d_h}}(\PP^{n_1}\times \cdots \times \PP^{n_h})\rangle$ (maybe reordering factors). In \cite[Theorem 3.1]{bb5} we proved that, if this is the case, then  $r_X(p)=\sum_{i=1}^h d_i$.

\item From now on we assume that $p\in \sigma _3(X)\setminus \sigma _2(X)$. By \cite[Theorem 1.2]{bl2} there is short list of zero-dimensional schemes $\Gamma\subset {\PP^{n_1}\times \cdots \times \PP ^{n_k}}$
such that $p\in \langle \nu _{n_1^{d_1},\dots ,n_k^{d_k}}(\Gamma)\rangle$, therefore, in order to prove   Theorem \ref{i1},  it is sufficient to bound the rank of the points in $\langle \nu _{n_1^{d_1},\dots ,n_k^{d_k}}(\Gamma)\rangle$ for each 
$\Gamma$ in their list. 

Since $p\in \sigma_3(X)\setminus \sigma _2(X)$, The possibilities for $\Gamma$ are only the following: either
$\Gamma$ is a smooth degree 3 zero-dimensional scheme (case \eqref{(a)} below), or it is the union of a degree 2 scheme supported at one point and a simple point (case \eqref{(b)}), or it is
a curvilinear degree $3$ scheme (case \eqref{(c)}) or, finally, a very particular degree $4$ scheme with $2$ connected components of degree $2$ (case \eqref{(d)}). 

\begin{enumerate}[(3a)]
\item\label{(a)}
If $\Gamma$ is a set of $3$ distinct points, then obviously $r_X({p}) =3$ (\cite[Case (i), Theorem 1.2]{bl2}). 

\item\label{(b)} If $\Gamma$ is a disjoint union of a  simple point $a$ and a degree $2$ connected scheme  
(\cite[Case (ii), Theorem 1.2]{bl2}), then there is a point  $q$ on a tangent line to $X$ such that $p\in \langle \{\nu _{n_1^{d_1},\dots ,n_k^{d_k}}(a),q\}\rangle$. Hence
$r_X({p})\le 1+r_X(q) \le 1+\sum _{i=1}^{k} d_i$ (for the rank on the tangential variety of $X$ see \cite{bb1}). Since $d_i>0$ for all $i$'s and $k\ge 2$, then $1+\sum _{i=1}^{k} d_i \le -1 +\sum _{i=1}^{k} 2d_i$. 

\item\label{(c)} Assume $\deg (\Gamma) =3$ and that $\Gamma$ is connected (\cite[Case (iii), Theorem 1.2]{bl2}) supported at a point $\{o\}:= \Gamma_{\mathrm{red}}$. 
Since the case $k=1$ is true by \cite[Theorem 37]{bgi}, we can prove the theorem by using induction on $k$, with the case $k=1$ as the starting case.

Since $\deg(\Gamma)=3$, by Corollary \ref{AutarkyCorollary}, we can assume that $p$ belongs to   a Segre-Veronese variety of $k$ factors all of them being either $\PP^1$'s or $\PP^2$'s, i.e., after having reordered the factors, 
$$\label{P1P2}
p\in\nu_{1^{d_1}, \ldots , 1^{d_a}, 2^{d_{a+1}}, \ldots 2^{d_k}}(\PP^1)^a \times ( \PP^2)^b.
$$
The $\PP^1$'s correspond to the cases in which 
either 
$\deg (\pi _i(\Gamma)) =3$ and $\dim \langle \pi _i(\Gamma)\rangle =1$ (i.e. $\pi _i(\Gamma)$ is contained in a line of the original $\PP^{n_i}$), 
or $\deg(\pi_i(\Gamma))=2$ (notice that in this case  $\pi _{i \mid \Gamma}$ is not an embedding). 
The $\PP^2$'s correspond to the cases in which $\dim \langle \pi_i(\Gamma)\rangle=2$, $\deg(\pi_i(\Gamma))=3$. Finally we can exclude all the cases in which $\deg(\pi_i(\Gamma))=1$ because, again by Corollary \ref{AutarkyCorollary}, we would have that $p$ belongs to a Segre-Veronse variety of less factors and then this won't give the highest bound for the rank of $p$.

Now fix a point $u\in \PP^1$. By Remarks \ref{Remarka1}, \ref{Remarka2} and \ref{Remarka3} there is 
\begin{equation}\label{fi}f_i: \PP^1 \to \PP^{n_i} \hbox{ with }f_i(3u) =\pi _i(\Gamma).
\end{equation} 
Consider the map 
$$\label{f1...fk}
f = (f_1,\dots ,f_k): \PP^1\to {\PP^{n_1}\times \cdots \times \PP ^{n_k}}.
$$
We have $f(u) = \{o\}$ and $\pi _i(f(3u)) =f_i(3u) =\pi _i(\Gamma)$. Since $\pi _i(f(3u)) =\pi _i(\Gamma)$ for all $i$'s, the universal property of products gives $f(3u) =\Gamma$. The map $f$ has multi-degree $(a_1,\dots ,a_k)$ where $a_i = 1$ if $n_i=1$ and $\deg (\pi _i(\Gamma)) =3$, and $a_i=2$ in all other cases.
Notice that $f_i$ is an embedding if $\deg (\pi _i(\Gamma)) \ne 2$. Since $\deg (\pi _i(\Gamma))=2$ if and only if $\pi _i^{-1}(o_i)$ contains the line spanned by the degree $2$ sub-scheme 
of $\Gamma$, we have $\deg (\pi _i(\Gamma))=2$ for at most one index $i$. Since $k\ge 2$, $f$ is an embedding. Set 
$$C:= \nu _{n_1^{d_1},\dots ,n_k^{d_k}}\left(f\left(\PP^1\right)\right) \hbox{ and } Z:= \nu _{n_1^{d_1},\dots ,n_k^{d_k}}(\Gamma).$$
The curve $C$ is smooth and rational of degree $\delta := \sum _{i=1}^{k} a_id_i$. Note that $\delta \le \sum _{i=1}^{k} 2d_i$. Hence to
prove Theorem \ref{i1} in this case it is sufficient to show that $r_C({p})\le \delta -1$ because clearly $r_C(p)\geq r_X(p)$ since $C\subset X$. 

By assumption $p\in \langle Z\rangle$. Since $p\notin \sigma _2(X)$, $\langle Z\rangle$ is not a line of $\PP^N$. Hence $\langle Z\rangle$ is a plane because $\deg(Z)=\deg(\Gamma)=3$. 
Since $C$ is a degree $\delta$ smooth rational curve, we have $ \dim \langle C\rangle \le \delta$. By \cite[Proposition 5.1]{lt} we have $r_C({p})\le  \dim \langle C\rangle$.
Hence it is sufficient to prove the case $\delta = \dim \langle C\rangle$, i.e. we may assume that $C$ is a rational normal curve in its linear span. 

\quad If $\delta \ge 4$, since $Z$ is connected and of degree $3$, by Sylvester's theorem (cf. \cite{cs}) we have $p$ has $C$-border rank $3$ and $r_C({p}) = \delta -1$, concluding the proof in this case. 

\quad If $\delta \le 3$,  we have $\sigma _2({C})=\langle C\rangle$ and hence $p\in \sigma _2(X)$, contradicting $p\in \sigma_3(X)\setminus \sigma_2(X)$.

\item\label{(d)} Assume that $\Gamma$ has degree $4$ (\cite[Case (iv), Theorem 1.2]{bl2}). J. Buczy\'nski and J.M. Landsberg show that $p$ belongs to the span of two tangent lines to $X$ whose set theoretic intersections with $X$ span a line which is contained in $X$. This means that $\Gamma = v\sqcup w$ with $v,w$ being two degree 2 reduced zero-dimensional schemes
 with support  contained in a line $L\subset {\PP^{n_1}\times \cdots \times \PP ^{n_k}}$ and moreover that the multi-degree of $L$ is  $\epsilon _i$ for some $i=1, \ldots , k$ (cfr. Notation \ref{Setting}). This case occurs only when $d_i=1$, i.e. when $\nu _{n_1^{d_1},\dots ,n_k^{d_k}}(L)=\nu _{n_1^{1},\dots ,n_k^{1}}(L)=\tilde L$
is a line.

Observe that $\tilde v:=\nu_{n_1^{d_1},\dots ,n_k^{d_k}}(v)$ and $\tilde w:=\nu_{n_1^{d_1},\dots ,n_k^{d_k}}(w)$ are two tangent vectors to $X$. In \cite[Theorem 1]{bb1} we prove that the $X$-rank of a point $p\in T_o(X) $  for  a certain point $o=(o_1, \ldots , o_k)\in X$, is the minimum number $\eta_{X}(p)$ for which  there exist  $E\subseteq \{1,\dots ,k\}$ such that $\sharp (E)=\eta_{X}(p)$ and $T_o(X) \subseteq \langle \cup _{i\in E} Y_{o,i}\rangle$ where $Y_{o,i}$ is the $n_i$-dimensional linear subspace obtained by fixing all coordinates $j\in \{1,\dots ,k\}\setminus \{i\}$ equal to $o_j\in \mathbb{P}^n_i$.
Let $I$ and $J$ be the sets playing the role of $E$ for $\langle\tilde v\rangle$ and $\langle\tilde w\rangle$ respectively
 and set  $ I'=I\setminus \{i\}$ (meaning that $I'=I$ if $i\notin I$ and $I'=I\setminus \{i\} $ otherwise) and $ J'=J\setminus\{i\}$ .
 Now take $$\alpha := \sum _{j\in I'} d_j + \sum _{j\in J'} d_j +d_i$$ and note that $\alpha \le -1 + \sum _{h=1}^{k} 2d_h$, therefore if we prove that $r_X(p)\leq \alpha$ we are done.
 Let $D_j\subset \PP^{n_1}\times \cdots \times \PP^{n_k}$, $j\in I'$, be the line of multi-degree $\epsilon _J$ containing $\pi _j( v)$, and let $T_j$, $j\in J'$, be the line of $X$ of multi-degree $\epsilon _j$ containing $\pi _j(w)$. The curve $ L\cup \left( \bigcup _{j\in I'} D_j\right)$ contains $v$ and the curve $ L\cup \left( \bigcup _{j\in J'} T_j\right)$ contains $ w$. Hence the curve 
 $$T:=  L\cup \left(\bigcup _{j\in I'} D_j\right)\cup \left(\bigcup _{j\in J'} T_j\right)$$
  is a reduced and connected curve containing $\Gamma$. Since $p\in \langle \nu _{n_1^{d_1},\dots ,n_k^{d_k}}(\Gamma)\rangle$, we have that if we call  $\tilde T:=\nu_{n_1^{d_1},\dots ,n_k^{d_k}}( T)$ then $p\in \langle\tilde T\rangle$ and $r_{X}({p}) \le r_{\tilde T}({p})$. The curve $\tilde T$
 is a  connected curve whose irreducible components are smooth rational curves and with $\deg (\tilde T)=\alpha$. Hence $\dim \langle \tilde T\rangle \le \alpha$. Since $\tilde T$ is reduced and connected, as in \cite[Proposition 4.1]{lt} and in \cite{cs}, we get $r_{\tilde T}({p})\le \alpha$. Summing up $r_X(p)\leq r_{\tilde T}(p)\leq \alpha\leq  -1 + \sum _{h=1}^{k} 2d_h$.
\end{enumerate}
\end{enumerate}
\end{proof}

\section{Proof  of Theorem \ref{dd1}}\label{SectionSegre}
{Autarky Lemma (proved in  \cite[Lemma 3.3]{bb5} and recalled here in Lemma \ref{autarky}) is true also for the border rank (\cite[Proposition 2.1]{bl1}). This allows to formulate the analog of Corollary \ref{AutarkyCorollary} for border rank. Therefore, in order to prove Theorem \ref{dd1} and $x\le k-1$, we can limit ourselves to the study of the case $n_i=1$ for all $i$'s. This is the reason why in the first part of this section we will always work with the Segre variety of  $\PP^1$'s. }
Let 
\begin{equation}\label{1k}\nu_{1^{(k)}} :(\PP^1 )^{k}\to \PP^r, \; r=2^k-1
\end{equation} be the Segre embedding of $k$ copies of $\PP^1$'s and $X:=\nu_{1^{(k)}}((\PP^1)^k)$; and let  
\begin{equation}\label{1k-1}\nu_{1^{(k-1)}}:(\PP^1)^{k-1}\to \PP^{r'}, \; r'=2^{k-1}-1\end{equation}
be the the Segre embedding of $k-1$ copies of $\PP^{1}$'s and $X':=\nu_{1^{(k-1)}}((\PP^1)^{k-1}).$

\begin{proposition}\label{d1}
Assume $k\ge 3$. 
Let $\Gamma\subset (\PP^1)^k$ be a degree $3$ connected curvilinear
scheme 
such that $\deg (\pi _i(\Gamma)) =3$ for all $i$'s, and let $\beta$ be the only degree $2$ sub-scheme of $\Gamma$. 
For all $p\in
\langle \nu_{1^{(k)}} (\Gamma)\rangle \setminus \langle \nu_{1^{(k)}} (\beta )\rangle$ we have that  
\begin{enumerate}[(a)]
\item if $k=3$, then  $2\le r_X({p})\le 3$ and $r_X({p})=2$ if $p$ is general in $\langle \nu_{1^{(k)}} (\Gamma)\rangle$; 
\item if $k\ge 4$, then $r_X({p}) =k-1$.
\end{enumerate}
\end{proposition}
  
\begin{proof} Since $\Gamma\subset (\PP^1)^k$ is connected, it has support at only one point; all along this proof we  set
\begin{equation}\label{o}o:=\mathrm{Supp}(\Gamma)\in (\PP^1)^k.\end{equation}

First of all recall that in step \eqref{(a)} of the proof of Theorem \ref{i1} we obtained an embedding $f = (f_1,\dots ,f_k)$ with $f_i: \PP^1\to \PP^1$ an isomorphism (see \eqref{fi}); moreover we can fix a point $u\in \PP^1$ such that $f(u) =o$ and $\Gamma = f(3u)$.
We proved that 
$$C:= \nu_{1^{(k)}} (f(\PP^1))$$
is a degree $k$ rational normal curve in its linear span. Obviously 
$$r_X({p}) \le r_C({p}).$$ 

If $k\ge 4$ Sylvester's theorem implies $r_C({p}) =k-1$. 

Now assume
$k=3$.  Since a degree
$3$ rational plane curve has a unique singular point, for any $q\in \langle C\rangle $ there is a unique line $L\subset \langle C\rangle =\PP^3 $
with $\deg (L\cap C)= 2$. Thus $r_C({p}) =2$ (resp. $r_C({p}) =3$) if and only if $p\notin \tau ({C})$ (resp. $p\in \tau ({C})$, cfr. Definition \ref{SecantTangent}). Since $\tau ({C})$ is a degree $4$ surface, by Riemann-Hurwitz,
we see that both cases occur and that $r_C({p})=2$ (and hence $r_X({p}) =2$ if $p$ is general in $\langle \nu_{1^{(k)}} (\Gamma)\rangle$). 

\begin{claim}\label{Claim1} Let the point $o\in (\mathbb{P}^1)^k$ be, as in \eqref{o}, the support of $\Gamma$.
Fix any $F\in |\Oo _{(\PP^1)^k}(\epsilon _k)|$ such that $o\notin F$. Then $\langle \nu_{1^{(k)}} (\Gamma)\rangle \cap \langle \nu_{1^{(k)}} (F)\rangle =\emptyset$.
\end{claim}

\begin{proof}[Proof of Claim \ref{Claim1}]
It is sufficient to show that $h^0(\Ii _{F\cup \Gamma}(1,\dots ,1)) = h^0(\Ii _F(1,\dots ,1)) -3$, i.e. $h^0(\Ii _\Gamma(1,\dots ,1,0)) = h^0(\Oo _{(\PP^1)^k}(1,\dots ,1,0))
-3$. This is true because $f_1, \ldots ,f_{k-1}$ (recalled at the beginning of the proof this Proposition \ref{d1} and introduced in \eqref{fi}) are isomorphisms.
\end{proof}

\begin{enumerate}[(a)]
\item\label{Prop(a)} Assume $k=3$. Since $r_X({p}) \le r_C({p}) \le 3$ and $r_C({p})=2$ for a general $p$ in $\langle \nu_{1^{(3)}} (\Gamma)\rangle$, we only need to prove that $r_X({p})>1$. 
The case  $r_X({p})=1$ corresponds to a completely decomposable tensor: $p=\nu_{1^{(3)}} (q)$ for some $q\in {(\PP^1)^3}$. Clearly $r_X(\nu_{1^{(3)}}(o))=1$ but $o\in \langle \beta \rangle$ then, since we took $p \in \langle \nu_{1^{(3)}} (\Gamma)\rangle \setminus \langle \nu_{1^{(3)}} (\beta )\rangle$, we have $p\neq \nu_{1^{(3)}}(o)$ and in particular $q\neq o$. 
In this case we can add $q$ to $\Gamma$ and get that  $h^1(\Ii _{q\cup \Gamma}(1,1,1)) >0$ by \cite[Lemma 1]{bb}. Since $\deg (f_i(\Gamma)) =3$, for all $i$'s, every point of $\langle \beta \rangle \setminus \{o\}$
has rank $2$.
Since $q:=(q_1,q_2,q_3)\ne o:=(o_1,o_2,o_3)$ we have $q_i\ne o_i$ for some $i$, say for $i=3$. Take $F\in |\Oo _{(\PP^1)^3}(\epsilon _3)|$
such that $q\in F$ and $o\notin F$. Hence $F\cap (\Gamma\cup \{q\}) = \{q\}$. We have $h^1(F,\Ii _{q,F}(1,1,1)) =0$, because $\Oo _{(\PP^1)^3}(1,1,1)$ is spanned. Claim \ref{Claim1} gives $h^1(\Ii _\Gamma(1,1,0)) =0$. The residual exact sequence of $F$ in ${(\PP^1)^3}$ gives $h^1(\mathcal{I}_{\Gamma\cup \{q\}}(1,1,1))=0$, a contradiction.

\item\label{Prop(b)}  From now on we assume $k\ge 4$ and that Proposition \ref{d1} is true for a smaller number of factors. Since $X\supset C$, we have $r_X({p}) \le k-1$ (in fact, as we already recalled above, $r_C({p}) =k-1$  by Sylvester's theorem). We need to prove that we actually have an equality, so we assume
$r_X({p})\le k-2$  and we will get a contradiction.

Take a set of points $S\in \Ss ({p})$  of $(\PP^1)^{k}$ evincing the $X$-rank of $p$ (see Definition \ref{DefRank}) and consider $v =(v_1,\dots ,v_k)\in S\subset (\PP^1)^{k}$ to be a point appearing in a decomposition of $p$. We can always assume that, if $o = (o_1,\dots ,o_k)$, then $v_k\ne o_k$: such a $v\in S\subset \mathcal{S}(p)$  exists because, by Autarky (here recalled in Lemma \ref{autarky}), no element of $\Ss ({p})$ is contained in $(\PP^1)^{k-1}\times \{o_k\}$.

Consider the pre-image  $$D:= \pi _k^{-1}(v_k).$$
Clearly by construction $o\notin D$ hence for any $q\in {(\PP^1)^k}\setminus D$
we have $h^1(\Ii _{q\cup D}(1,\dots ,1)) = h^1(\Ii _{q}(1,\dots ,1,0)) = 0$, because $\Oo _{(\PP^1)^k}(1,\dots ,1,0)$ is globally generated. 
This implies that  $\langle \nu_{1^{(k)}} (D)\rangle$ intersects $X$ only in $\nu_{1^{(k)}} (D)$.

Now consider $$\ell :\PP^{2^{k}-1}\setminus \langle \nu_{1^{(k)}} (D)\rangle \to \PP^{2^{k-1}-1}$$  the linear projection from $ \langle \nu_{1^{(k)}} (D)\rangle$. 
Since $p \notin \langle \nu_{1^{(k)}} (D)\rangle$ (Claim \ref{Claim1}), $\ell$ is defined at $p$. 
Moreover the map $\ell $ induces a rational map $\nu_{1^{(k)}} ({(\PP^1)^k}\setminus D)\to \nu_{1^{(k-1)}} ({(\PP^1)^{k-1}})$ which is induced by the projection $\tau _k: {(\PP^1)^k}\to {(\PP^1)^{k-1}}$ defined in Notation \ref{Setting}.  We have 
$$\ell \circ \nu_{1^{(k)}} = \nu_{1^{(k-1)}}\circ \tau_k.$$
Since $o\notin D$, we have $\ell (\langle \Gamma\rangle ) =\langle \nu_{1^{(k-1)}}(\Gamma')\rangle$, where $\Gamma' =\tau _k(\Gamma)$. Hence $p':= \ell ({p})\in \langle \nu_{1^{(k-1)}}(\Gamma')\rangle$. By \cite{bb1} every element of $\langle \nu_{1^{(k-1)}}(\beta )\rangle \setminus \nu_{1^{(k-1)}}(o')$, with  $o':= \tau _k(o)$, has $X'$-rank
$k-1$.
Since $\deg (\pi _i(\Gamma)) =3$ for all $i$'s, we have $\deg (\pi _i(\beta ))=2$ for $i=1,\dots ,k-1$. This implies that the minimal sub-scheme  $\alpha$  of    $\Gamma'$ such that $p'\in \langle \nu_{1^{(k-1)}}(\alpha)\rangle$  is such that  $\alpha \ne \beta$ where  $\beta$ is the degree $2$ sub-scheme of $\Gamma'$. Now let $S'\subset (\PP^{1})^{k-1}$ be the projection by $\tau_k$ of the
set of points of $S\subset \mathcal{S}(p)$ that are not in $D$, i.e. 
$S':= \tau _k(S\setminus S\cap D)$. 
Since $\sharp (S')\le k-2$ and $p'\in \langle \nu_{1^{(k-1)}}(\Gamma')\rangle$, the inductive assumption gives $\alpha \ne \Gamma'$ (it works even when $k=4$). Hence $\alpha =\{o'\}$. Thus
$p\in \langle \nu_{1^{(k)}} (\{o\}\cup D)\rangle$. Hence $\dim (\langle \nu_{1^{(k)}} (\Gamma \cup D)\rangle) \le \dim ( \langle \nu_{1^{(k)}} (D)\rangle) +2$, contradicting Claim \ref{Claim1}.
\end{enumerate}\end{proof}

We need the following lemma, which is the projective version of an obvious linear algebra exercise.

\begin{lemma}\label{dd1.101}
Fix two linear spaces $L_1\subsetneq L_2\subset \PP^m$ and a finite set $E\subset L_2$ spanning $L_2$. Let $\ell : \PP^m\setminus L_1\to \PP^z$, $z:=
m -1-\dim L_1$, be the linear projection from $L_1$. Then $\ell (L_2\setminus L_1)$ is a linear space spanned by the set $\ell (E\setminus E\cap L_1)$.
\end{lemma}

\begin{notation}
Fix $(a,b)\in \NN ^2\setminus \{(0,0)\}$. Let $\Delta _{a,b}$ be the set of all pairs $(f,o)$, where $o\in \PP^1$, $f: \PP^1\to (\PP^2)^a\times (\PP ^1)^b$, each $\pi _i\circ f$, $1\le i\le a$, is a degree
$2$ embedding and, for $a+1\le i\le b$, $\pi _i\circ f$ is an isomorphism.
\end{notation}

\begin{lemma}\label{uu1} Set $\tilde G=\mathrm{Aut}(\mathbb{P}^2)^a \times \mathrm{Aut}(\PP^1)^b$, $G:= \tilde G \times \mathrm{Aut}(\PP^1)$. Let $G$ acts on $\Delta_{a,b}$ via
$(g,h) (f,o)=(g \circ f \circ h^{-1}, h(o))$.
Then this action is transitive, i.e., for $(f,o), (f',o')$ we have $(g,h)\in G$ such that $h(o)=o'$ and $g\circ f \circ h^{-1} = f'$.
\end{lemma}

\begin{proof}
Fix any $h\in \mathrm{Aut}(\PP^1)$ such that $h(o) = o'$ and write $\tilde{f}:= f\circ h^{-1}$. \\Write $\tilde{f }=(\tilde{f_1},\ldots ,\tilde{f}_a , \tilde{f}_{a+1}, \ldots, \tilde{f}_{a+b})$ and $f' = (f'_1,\ldots ,f'_a, f'_{a+1}, \ldots ,f'_{a+b})$ with $\tilde{f_i}:=\pi_i\circ \tilde{f}$ and $f'_i:=\pi_i\circ f'$. We need to find $g = (g_1,\dots ,g_a,g_{a+1},\dots ,g_{a+b})\in \tilde G$
such that $g\circ \tilde{f} = f'$, i.e. by the universal property of maps to products, we need to find $g= (g_1,\dots ,g_a,g_{a+1},\dots ,g_{a+b})\in \tilde G$ such that
$g_i\circ \tilde{f}_i = f'_i$ for all $i$.

If $a+1\le i \leq a+b$ take $g_i:= f'_i\circ \tilde{f_i}^{-1}$.  

Now we fix $i$ such that  $1\leq i \leq a$. We have two degree $2$ embeddings $f'_i: \PP^1\to \PP^2$ and $\tilde{f_i} : \PP^1\to \PP^2$. Any two such
maps are equivalent, up to an automorphism of $\PP^2$, because these embeddings are induced by  the complete linear system of the anticanonical line bundle of $\PP^1$.
Thus there is $g_i\in \mathrm{Aut}(\PP^2)$ such that $g_i\circ \tilde f_i =f'_i$.
\end{proof}

\begin{notation}\label{star}
Take $Y = (\PP^2)^a\times (\PP^1)^b$ and let $\nu_{2^{(a)},1^{(b)}}: Y \to \mathbb{P}^N $, $N:=3^a2^b-1$, be the Segre embedding of $Y$. Let $\Gamma _{a,b}$ (resp. $\Gamma '_{a,b}$) be the set of all $p\in \PP^N$, such there is $(f,o) \in \Delta _{a,b}$ with $p\in \langle \nu_{2^{(a)},1^{(b)}} (f(3o))\rangle$ (resp. and $p\notin \langle \nu_{2^{(a)},1^{(b)}} (f(2o))\rangle$). 
\end{notation}

Since the image of an algebraic set by a morphism is constructible, $\Gamma _{a,b}$ and $\Gamma '_{a,b}$ of Notation \ref{star} are constructible sets. The closure of $\Gamma _{a,b}$ in $\PP^N$ is irreducible.
Therefore we are allowed to inquire about the rank of a general element of $\Gamma _{a,b}$. If either $a>0$ or $b\ge 2$, then
$\Gamma '_{a,b}\ne \emptyset$ and the closures in $\PP^N$ of $\Gamma _{a,b}$ and $\Gamma '_{a,b}$ are the same.

\begin{lemma}\label{uu2}
For all $k\ge 3$ we have $r_X({p}) =2k-1$ for a general $p\in \Gamma _{k,0}$ as in Notation \ref{star}.
\end{lemma}

\begin{proof}
 We use induction on $k$, the case $k=3$ being true by \cite[Theorem 1.8]{bl2}. 
\\Now assume $k\ge 4$. Call $\nu_{2^{(k)}} :
(\PP^2)^k\to \PP ^r$, $r:= 3^k-1$,  the Segre embedding. Fix $a\in \PP^1$. For each $1\le i \le k$ let $f_i:\PP^1\to \PP^2$ be a
degree $2$ embedding. Let $f =(f_1,\dots ,f_k): \PP ^1\to (\PP^2)^k$ be the embedding with $f_i=\pi _i\circ f$ for all $i$. As in step
(\ref{(c)}) of the proof of Theorem \ref{i1} we see that the curve $C:= \nu_{2^{(k)}} (f(\PP^1))$ is a rational normal curve of degree
$2k$ in its linear span. Fix $a\in \PP^1$ and set $o:= (o_1,\dots ,o_k):= f(a)$ and $A:= f(3a)$. The scheme $\nu_{2^{(k)}} (A)$ has degree $3$ and it is
curvilinear. Fix a general $p\in \langle \nu_{2^{(k)}} (A)\rangle \setminus \langle \nu_{2^{(k)}} (2o)\rangle$. Since $p$ has border rank
$3$ with respect to the rational normal curve $C$, Sylvester's theorem gives $r_C({p}) = 2k-1$. Hence $r_X({p})\le 2k-1$. To
prove the lemma for the integer $k$ it is sufficient to prove that $r_X({p})\ge 2k-1$.

Assume $r_X({p})\le 2k-2$ and
fix $B\in \Ss ({p})$.

\begin{enumerate}[(a)]
 \item\label{quad(a)} In this step we assume the existence of a line $L\subset \PP^{n_k}$ such that $o_k\notin L$ and $\sharp (Y'\cap B)\ge 2$, where $Y':= \PP^{n_1}\times \cdots \times \PP^{n_{k-1}}\times L$. We have $Y'\in |\Oo _Y(\epsilon _k)|$. Since $o_k\notin L$, we have $o\notin Y'$ and hence $A\cap Y' =\emptyset$. Set $B':= B\setminus B\cap Y'$. Set $A':= \tau_k(A)$ where $\tau_k$ is defined in Notation \ref{Setting}.  Since $k\ge 3$ and $(f_1,f_2): \PP^1\to \PP^{2}\times \PP^{2}$ is an embedding, we have $\deg (A')=3$.
Let $\nu_{2^{(k-1)}}: (\mathbb{P}^2)^{k-1}\to \PP^s$, $s=3^{k-1}-1$, be the Segre embedding of $(\mathbb{P}^2)^{k-1}$. Note that the linear projection from $L$ of $\PP^2\setminus L$
sends $\PP^2\setminus L$ onto a point. Set $E:= \langle \nu_{2^{(k)}} (Y')\rangle$. We have $\dim E = 2\cdot 3^{k-1}-1$. Let $\ell :
\PP^M\setminus E\to \PP^s$ denote the linear projection from $E$. Since $A\cap Y' =\emptyset$, $\ell (\nu_{2^{(k)}} (A))$ is a
well-defined zero-dimensional scheme.  Note that $\nu _{2^1,2^1}(f_1,f_2)(\PP^1)$ is not a line of the Segre embedding of
$\PP^{2}\times \PP^{2}$. Since $k\ge 3$, we get that $\nu_{2^{(k-1)}}(A')$ spans a plane. Hence $\ell (\nu_{2^{(k)}} (A)) =A'$ is linearly
independent, i.e. $\langle \nu_{2^{(k)}} (A)\rangle \cap E=\emptyset$. Hence $p':= \ell ({p})$ is well-defined and in particular it is
well-defined its rank with respect to the Segre variety $X':= \nu_{2^{(k-1)}}((\mathbb{P}^2)^{k-1})$. Since $\dim \langle \nu_{2^{(k)}} (A)\rangle = \dim \langle
\nu_{2^{(k)}} (A')\rangle$ and $p$ is general in $\langle \nu_{2^{(k)}} (A)\rangle$, $p'$ is general in $\langle \nu_{2^{(k-1)}}(A')\rangle$. By the
inductive assumption (case $k\ge 5$) or by \cite[Theorem 1.8]{bl2} (case $k=4$), we have $r_{X'}(p') =2k-3$. Since $p\in
\langle \nu_{2^{(k)}} (B)\rangle$, Lemma \ref{dd1.101} applied to $E:= \nu_{2^{(k)}} (B)$, $m=3^k-1$ and $L_1=E$, gives $p'\in \langle \nu_{2^{(k-1)}}(B')\rangle$. Since $\sharp (B')\le \sharp (B)-2 < 3k-3$, we get a contradiction.

\item \label{quad(b)} Assume the non-existence of a line $L\subset \PP^{n_k}$ such that $o_k\notin L$ and $\sharp (Y'\cap B)\ge 2$. By Autarky we have $B\nsubseteq (\mathbb{P}^2)^{k-1}\times \{o_k\}$.
Hence the assumption of this step is equivalent to assuming the existence of $b\in B$ such that $\pi _k(b) \ne o_k$, but $\pi _k(B)$ is contained in the line
$R\subset \PP^{n_2}$ spanned by $o_k$ and $\pi _k(b)$. Hence $B\subset (\mathbb{P}^2)^{k-1}\times R$, contradicting Autarky, because $n_k=2$ and $f_k(3a)$ spans $\PP^2$.
\end{enumerate}
\end{proof}

\begin{lemma}\label{uu4} Let $\nu_{2^{(1)},1^{(1)}} (Y)$ be the Segre embedding of $\PP^2\times \PP^1$.
We have $\Gamma_{1,1} \nsubseteq \nu_{2^{(1)},1^{(1)}} (Y)$.
\end{lemma}

\begin{proof}
We have $\tau (\nu_{2^{(1)},1^{(1)}} (Y)) \supsetneq \nu_{2^{(1)},1^{(1)}} (Y)$. Since a general tangent vector of $Y$ is of the form $f(2o)$
with $(f,o)\in \Delta _{1,1}$, we get $\Gamma _{1,1}\nsubseteq \nu_{2^{(1)},1^{(1)}} (Y)$.
\end{proof}

\begin{definition} Let $X\subset \mathbb{P}^N$ be any variety,  $Z$ a zero-dimensional scheme and $H$  an effective Cartier divisor. We define the scheme $Res _H(Z)\subset \mathbb{P}^N$ to be  the \emph{residue scheme of $Z$ with respect to $H$}, namely the subscheme of $\mathbb{P}^N$ whose ideal sheaf is $\mathcal{I}_Z : \mathcal{I}_H$.
\end{definition}

\begin{lemma}\label{uu5}
Take $Y = (\PP^2)^2$. For every $p\in \Gamma '_{2,0}$ we have $r_X({p}) >2$ (cf. Notation \ref{star}).
\end{lemma}

\begin{proof}
Assume the existence of a set $B\subset Y$ such that $\sharp (B)\le 2$ and $p\in \langle \nu_{2^{(2)}}(B)\rangle$. Since $B\in \Ss ({p})$, we have $p\notin \langle \nu_{2^{(2)}}(B')\rangle$ for
any $B'\subsetneq B$. Take $(f,o)\in \Delta _{2,0}$ such that $p\in \langle \nu_{2^{(2)}}(f(3o))\rangle$
and $p\notin \langle f(2o)\rangle$. Set $A:= f(3o)$. By assumption we have $p\notin \langle \nu_{2^{(2)}}(A')\rangle$ for any $A'\subsetneq A$. In particular $B \ne \{o\}$. By \cite[Lemma 1]{bb} we have $h^1(\Ii _{A\cup B}(1,1)) >0$. Since $\sharp (B)\le 2$, there is a line $R\subset \PP^2$ such that $\pi _1(B)\subset
R$. Set $H:= R\times \PP^2 \in |\Oo _Y(1,0)|$ and call $\nu ': H\to \PP^5$ the Segre embedding of $H$.
We have $\mathrm{Res} _H(A\cup B) \subseteq A$. Since $\pi _2(A)$ spans $\PP^2$ by the definition of $\Gamma _{2,0}$, $\pi _{2|\mathrm{Res} _H(A\cup B)}$ is an embedding
and  $\pi _2(A\cup B)$ is linearly independent. The residual exact sequence of $H$ in $Y$ gives $h^1(H,\Ii _{H\cap (A\cup B),H}(1,1)) >0$. Hence $\langle \nu '(H\cap A)\rangle
\cap \langle \nu '(H\cap B)\rangle \ne \emptyset$. Since $\pi _1(A)$ spans $\PP^2$, we have $A\nsubseteq H$. Thus $H\cap A \subsetneq A$. By the definition of
$\Gamma '_{2,0}$ we have $p\notin \langle \nu_{2^{(2)}}(H\cap A)\rangle$. Set $J: = \langle
\nu_{2^{(2)}}(A)\rangle \cap \nu_{2^{(2)}}(Y)$. Since the only linear subspaces of $\nu_{2^{(2)}}(Y)$ are the ones contained in a ruling of $Y$ and $(f,o)\in
\Delta _{2,0}$, the plane
$\langle\nu_{2^{(2)}}(A)\rangle$ is not contained in $\nu_{2^{(2)}}(Y)$. Hence $J\nsubseteq \langle \nu_{2^{(2)}}(A)\rangle$. Since $\nu_{2^{(2)}}(Y)$ is
scheme-theoretically cut out by quadrics, $J$ is cut out by plane conics. Write $J=\nu_{2^{(2)}}(I)$ with $I\subset Y$. $J$ is not a
reducible conic or a double line or a line, because $\pi _i(A)$ spans $\PP^2$, $i=1,2$, while all linear subspaces of $\nu_{2^{(2)}}
(Y)$ are contained in a ruling of $Y$.
If $J$ were a smooth conic we would have that either
$\pi _1(I)$ spans
$\PP^2$ and $\pi _2(I)$ is a point, or $\pi _2(I)$ spans $\PP^2$ and $\pi _1(I)$ is a point or $\pi _1(I)$ and $\pi _2(I)$
are lines, contradicting the assumption that each $\pi _i(A)$ spans $\PP^2$. Thus $J$ is a zero-dimensional scheme of
degree $\le 4$. Since $A\cup B\subseteq I$, we get that either $B = \{o\}$ (and we excluded this case) or $B= \{o,q\}$ for
some $q\in A$ with $q\ne o$. Thus $\deg
(A\cup B) =4$. We have
$h^1(\Ii _{A\cup B}(1,1))
\ne 0$ (\cite[Lemma 1]{bb}).  Since $p$ has not rank $2$ with respect to $\nu_{2^{(2)}}(C)$, we have $q\notin C$. Thus there
is $M\in |\Oo_Y(1,1)|$ with $M\supset C$ and $q\notin M$. Thus $M\cap (A\cup B) =A$ and $\mathrm{Res} _M(A\cup B) =\{q\}$. Thus
$h^1(\Ii_Q)=0$. Since
$h^1(\Ii _A(1,1)) =0$, the residual exact sequence of $M$ in $Y$ gives a contradiction. 
\end{proof}

\begin{lemma}\label{uu3}
Fix integers $a \ge 0$ and $b\ge 0$ with $a+b \ge 3$. We have $r_X({p}) =2a+b-1$ for a general $p\in \Gamma _{a,b}$ (cf. Notation \ref{star}).
\end{lemma}

\begin{proof}
The case $a=0$ is true by Proposition \ref{d1}. The case $b=0$ is true by Lemma \ref{uu2}. Thus we may assume that $a>0$ and $b>0$. Set $k:= a+b$. Take $(f,o)\in \Delta _{a,b}$ such that $p$ is a general element of $\langle \nu_{2^{(a)},1^{(b)}} (A)\rangle$ with $A:= 3o$. Set $C:= f(\PP^1)$, $f_i:= \pi _i\circ f$ and $o_i:= \pi _i(f(o))$. Since $\nu_{2^{(a)},1^{(b)}} ({C})$ is a degree $2a+b$ rational normal
curve in its linear span and $2a+b \ge 4$, Sylvester's theorem gives $r_{\nu_{2^{(a)},1^{(b)}} ({C})} =2a+b-1$. Thus $r_X({p}) \le 2a+b-1$. Assume $r_X({p}) \le 2a+b-2$ and take
$B\in \Ss ({p})$. By Autarky we
have $B\nsubseteq (\PP^2)^a\times (\PP^1)^{b-1}\times \{o_k\}$. Take $z\in B$ such that $b_k:= \pi _k(z)\ne o_k$.
Set $Y':=(\PP^2)^a\times (\PP^1)^{b-1}\times \{b_k\}$. Let $\nu_{2^{(a)},1^{(b-1)}}:= (\PP^2)^a\times (\PP^1)^{b-1}\to \PP^s$, $s:= -1 +3^a2^{b-1}$,
be the Segre embedding of $(\PP^2)^a\times (\PP^1)^{b-1}$. Set $E:= \langle \nu_{2^{(a)},1^{(b)}} (Y')\rangle$. We have $\dim E+1 = 2\cdot 3^2$.
Let $\ell : \PP^M\setminus E\to \PP^s$ the linear projection from $E$. Set $A':= \tau_k(A)$ (as in Notation \ref{Setting}).  As in the proof of Lemma \ref{uu2}
we get $E\cap \langle \nu_{2^{(a)},1^{(b)}} (A)\rangle =\emptyset$, $\nu_{2^{(a)},1^{(b-1)}} (A')=\ell (A)$, $p':= \ell ({p})$ is a general element of $\langle \nu_{2^{(a)},1^{(b-1)}}
(A')\rangle$. 

\begin{enumerate}[(a)]

\item  Assume $(a,b) =(1,2)$.  Since $\nu_{2^1,1^2} (Y)\nsubseteq \Gamma _{1,2}$, $p$ is general in $\Gamma _{1,2}$ and $\sharp (B) \le 2$, we have $\sharp (B)=2$.
Thus $\sharp (A')=1$ and so $p' \in \nu_{2^{(1)},1^{(1)}} (\PP^2\times \PP^1)$. Hence a general element of $\Gamma _{1,1}$ has rank $1$, contradicting Lemma \ref{uu4}.

\item Assume $(a,b) =(2,1)$. We use Lemma \ref{uu5}.

\item By the previous steps we may assume $a+b\ge 4$, $a>0$, $b>0$ and use induction on the integer $a+b$. (and hence by the inductive assumption applied to $(a,b-1)$ it has rank $2a+b-2$), while $p'\in \langle \nu_{2^{(a)},1^{(b-1)}}(B\setminus B\cap Y')\rangle$
with $\sharp (B\setminus B\cap Y')\le x-2$ (because $b_k\in \pi _k(Y')\cap \pi _k(B)$), a contradiction.
\end{enumerate}
\end{proof}

\begin{proof}[Proof of Theorem \ref{dd1}:]
First assume $x\le k-1$. If $x = 3$, then we may take as $p$ a general point of $\sigma _3(X)$. Now assume $x\ge 4$ and hence $k\ge 5$. Apply Proposition \ref{d1} to $(\PP^1)^{x+1}$ and then
use Autarky (Lemma \ref{autarky}).
Now assume $k\le x \le 2k-1$. For $x=2k-1$ use Lemma \ref{uu2} and Autarky. For each $x\in \{4,\dots ,2k-2\}$ use the case $a = x+1-k$ and $b= k-a$ of Lemma \ref{uu3} and then
apply Autarky.\end{proof}

\begin{remark}\label{uu6}
Take the set-up of Theorem \ref{dd1}. If $n_i\ge 2$ for all $i$, then Theorem \ref{dd1} gives all ranks of points of $\sigma _3(X)\setminus \sigma _2(X)$, but it does not say the rank
of each point of $\sigma _3(X)\setminus \sigma _2(X)$. One problem is that in Lemma \ref{uu2} we do not check all ranks of points of $\Gamma '_{1,1}$. A bigger problem
is that the inductive proof should be adapted and the induction must start. These problems may be not deal-breakers, but there is a class of points
of $\sigma _3(X)\setminus \sigma _2(X)$ (occurring even if $n_i=1$ for some $i$) for which we do not have a good upper bound for the rank (except that $r_X({p})\le 2k-1$). These are the points $p\in \langle \nu_{n_1^1, \ldots , n_k^1} (A)\rangle$ with $A\subset Y$ a connected curvilinear scheme of degree $3$ and $\deg (\pi _i(A)) = 2$ for some $i$, because
in this case $A\nsubseteq C$ with $C\subset Y$ and $\nu_{n_1^1, \ldots , n_k^1} ({C})$ a rational normal curve in its linear span. We have no idea about the rank of these points.
\end{remark}

 \section{Proof of Theorem \ref{i2}}\label{SectionCurvilinear}

 \begin{lemma}\label{c1}
 Fix an integer $c>0$ and $u\in \PP^1$. Let $E =cu \subset \PP^1$ be the degree $c$ effective divisor of $\PP^1$ with $u$ as its support. Let $g: E\to \PP^n$ be any morphism.
 Then there is a non-negative integer  $e\le c$ and a morphism $h: \PP^1\to \PP^n$ such that $h^\ast \left(\Oo _{\PP^n}(1)\right) \cong \Oo _{\PP^1}(e)$ and $h_{|E} =g$.
 \end{lemma}
 
 \begin{proof} Every line bundle on $E$ is trivial. We fix an isomorphism between $g^\ast \left(\Oo _{\PP^n}(1)\right)$ and $\Oo _E({c})$. After this identification, $g$ is induced
 by $n+1$ sections $u_0,\dots ,u_n$ of $\Oo _E({c})$ such that at least one of them has a non-zero restriction at $\{u\}$. The map $H^0\left(\Oo _{\PP^1}({c})\right)\to H^0\left(\Oo _E({c})\right)$
 is surjective and its kernel is the section associated to the divisor $cu$. Hence there are $v_0,\dots ,v_n\in H^0\left(\Oo _{\PP^1}({c})\right)$ with $v_{i|E} =u_i$ for all $i$.
 Not all sections $v_0,\dots ,v_n$ vanish at $0$. If they have no common zero, then they define a morphism $\PP^1\to \PP^n$ extending $g$ and we may take $e=c$.
 Now assume that they have a base locus and call $F$ the scheme-theoretic base locus of the linear system associated to $v_0,\dots ,v_m$. We have $\deg (F) \le c$. Set $e:= c-\deg (F)$ and $S:= F_{\mathrm{red}}$. The sections $v_0,\dots ,v_n$ induce a morphism $f: \PP^1\setminus S\to \PP^n$ with $f_{|E}=g$. See $v_0,\dots ,v_n$ as elements of $|\Oo _{\PP^1}({c})|$
and set $u_i:= u-F\in |\Oo _{\PP^1}(e)|$. By construction the linear system spanned by $u_0,\dots ,u_n$ has no base points, hence it induces a morphism
$h: \PP^1\to \PP^n$ such that $h^\ast \left(\Oo _{\PP^n}(1)\right) \cong \Oo _{\PP^1}(e)$. We have $h_{|\PP^1\setminus S} = f$ and hence $h_{|E} =g$.
 \end{proof}
 
 \begin{proof}[Proof of Theorem \ref{i2}:]  Let $Z\subset {\PP^{n_1}\times \cdots \times \PP ^{n_k}}$ such that $p\in \langle \nu _{n_1^{d_1},\dots ,n_k^{d_k}}(Z)\rangle$ and $Z$ has $Z_1$, $\ldots$ , $Z_\alpha$ connected components,
By assumption there is $p_i\in \langle \nu _{n_1^{d_1},\dots ,n_k^{d_k}}(Z_i)\rangle$ such that $p\in \langle \{p_1,\dots ,p_\alpha\}\rangle$. Note that if Theorem \ref{i2} is true for each $(Z_i,p_i)$, then it is true for $Z$. Hence it is sufficient to prove Theorem \ref{i2} under the additional assumption that
 $Z$ is connected, so from now on we assume
 \begin{itemize} \item $Z$ connected. \end{itemize}
Moreover,  since $r_X({p}) =1 \le 2-1+\sum _i d_i$ if $c=1$, we may  also assume that  
 \begin{itemize} \item
$\deg Z=c\ge 2$.
 \end{itemize}
Finally, since the real-valued function $x\mapsto x\left(-1+\sum _{i=1}^{k}d_i\right)$
 is increasing for $x\ge 1$, with no loss of generality we may assume that, for any $G\subsetneq Z$,
 \begin{itemize} \item
$p\notin \langle \nu _{n_1^{d_1},\dots ,n_k^{d_k}}(G)\rangle$.
 \end{itemize}
 Fix $u\in \PP^1$ and let $E =cu \subset \PP^1$ be the degree $c$ effective divisor of $\PP^1$ with $u$ as its support. Since $Z$ is curvilinear and
 $\deg (Z) =c$, we have $Z\cong E$ as abstract zero-dimensional schemes. Let $g: E\to {\PP^{n_1}\times \cdots \times \PP ^{n_k}}$ be the composition of an isomorphism $E\to Z$ with the inclusion
 $Z\hookrightarrow {\PP^{n_1}\times \cdots \times \PP ^{n_k}}$:
 $$g: E\to Z\hookrightarrow  {\PP^{n_1}\times \cdots \times \PP ^{n_k}}.$$
 Set $g_i:= \pi _i\circ g$. If we apply Lemma \ref{c1} to each $g_i$, we get the existence of an integer $c_i\in \{0,\dots ,c\}$ and of a morphism $h_i: \PP^1\to \PP^{n_i}$ such
 that $h_{i|Z} = g_i$ and ${h_i}^\ast \left(\Oo _{\PP^{n_i}}(1)\right) \cong \Oo _{\PP^1}(c_i)$. The map \begin{equation}\label{h}
h = (h_1,\dots ,h_k):\PP^1 \to {\PP^{n_1}\times \cdots \times \PP ^{n_k}}
\end{equation} 
 has multi-degree $(c_1,\dots ,c_k)$.
 The curve $$D:= h(\PP^1)$$ is an integral rational curve containing $Z$. Since $p\in \langle \nu _{n_1^{d_1},\dots ,n_k^{d_k}}(Z)\rangle$, we have $$r_X({p})\le r_{\nu _{n_1^{d_1},\dots ,n_k^{d_k}}(D)}({p}).$$
 Thus it is sufficient to prove that, if we call $\tilde D:= \nu _{n_1^{d_1},\dots ,n_k^{d_k}}(D)$, then
 $r_{\tilde D}({p}) \le 2+c\left(-1+\sum _{i=1}^{k} d_i\right)$. Since $c_i\le c$ for all $i$, it is sufficient to prove
 that $r_{\tilde D}({p}) \le 2-c+\sum _{i=1}^{k} c_id_i$. 
 
 Set  ${\tilde Z}:= \nu _{n_1^{d_1},\dots ,n_k^{d_k}}(Z)$, $m:= \dim (\langle \tilde D\rangle)$ and 
 $$f = \nu _{n_1^{d_1},\dots ,n_k^{d_k}}\circ h :\PP^1\to\PP^N.$$  
 By assumption ${\tilde Z}$ is linearly independent in $\langle \tilde D\rangle \cong \PP^m$ and
 in particular $c\le m+1$.
 
 \begin{enumerate}[(a)]
 \item\label{a} Assume that the map $h$ defined in \eqref{h} is birational onto its image. The curve $\tilde D \subset \PP^N$ just defined is a rational curve of degree $a:= \sum _{i=1}^{k} c_id_i$ contained in the projective space $\PP^m:=
 \langle \tilde D\rangle$ and non-degenerate in $\PP^m$. Note that $a\ge m$ and that $p\in \langle {\tilde Z}\rangle$.
 
 \begin{enumerate}[(1)]
 \item\label{a1} First assume that $a=m$. In this case $\tilde D$ is a rational normal curve of $\PP^m$. If $c\le \left\lceil (a+1)/2\right\rceil$,
 then Sylvester's theorem implies that $r_{\tilde D}({p}) = a+2-c =2-c+\sum _{i=1}^{k} c_id_i$.  Now assume $c > \left\lceil (a+1)/2\right\rceil$. Since ${\tilde Z}$ is connected and curvilinear and $p\notin \langle G\rangle$ for any $G\subsetneq {\tilde Z}$, Sylvester's theorem implies $r_{\tilde D}({p}) \le c$. 
 
 \item\label{a2} Now assume $m<a$. There is a rational normal curve $C\subset \PP^a$ and a linear subspace $W\subset \PP^a$ such that $\dim (W) =a-m-1$, $C\cap W =\emptyset$
 and $h$ is the composition of the degree $a$ complete embedding $j:= \PP^1\hookrightarrow \PP^a$ and the linear projection $\ell : \PP^a \setminus W\to \PP^m$ from $W$.
 The scheme $E':= j(E)$ is a degree $c$ curvilinear scheme and $\ell$ maps $E'$ isomorphically onto ${\tilde Z}$. Since ${\tilde Z}$ is linearly independent, then $\langle E'\rangle \cap W=\emptyset$ and $\ell$ maps isomorphically $\langle E'\rangle$ onto $\langle {\tilde Z}\rangle$. Thus there is a unique $q\in \langle E'\rangle$ such that $\ell (q) =p$. Take any finite set $S\subset j(\PP^1)$
 with $q\in \langle S\rangle$. Since $C\cap W =\emptyset$, $\ell (S)$ is a well-defined subset of ${\tilde D}$ with cardinality $\le \sharp (S)$. Hence $r_{\tilde D}({p})\le r_C(q)$.
 As in step \eqref{a1} we see that either $r_C(q) =a+2-c$ (case $c \le \left\lceil (a+1)/2\right\rceil$) or $r_C(q) \le c$ (case $c > \left\lceil (a+1)/2\right\rceil$).
 \end{enumerate}
 
 \item\label{b} Now assume that $h$ is not birational onto its image, but it has degree $k\ge 2$. Note that $k$ divides $c_i$ for all $i$. In this case we will prove that $r_{\tilde D}({p}) \le 2-c+\sum _{i=1}^{k} c_id_i/k$. Let $h': \PP^1 \to h(\PP^1)$ denote the normalization map. There is a degree $k$ map $h'': \PP^1\rightarrow \PP^1$ such that
 $h$ is the composition of $h'\circ h''$ and the inclusion $h(\PP^1)\subset {\PP^{n_1}\times \cdots \times \PP ^{n_k}}$. We have $Z = h'(E')$, where $E' =cu'$ and $u' =h''(u)$. We use $E'$ and $h'$ instead of $E$ and $h$ and repeat verbatim step \eqref{a}.
 \end{enumerate} \end{proof}

\section*{Acknowledgements} We want to thank the anonymous referee and the Handling Editor Jan Draisma for their careful jobs that improved the presentation of this paper. A special thank to the referee for her/his  very interesting questions that encouraged us in giving a better version of Theorem \ref{dd1}.

\end{document}